\documentclass[11pt,reqno]{amsart}
\usepackage{graphicx,amssymb,mathrsfs,amsmath,amsfonts,appendix}
\usepackage{subfigure}
\usepackage{amsthm}
\usepackage{paralist}
\usepackage{enumitem}
\usepackage[TS1,T1]{fontenc}
\usepackage[utf8]{inputenc}
\usepackage{comment}
\usepackage{amsmath}
\usepackage{enumitem}
\usepackage{mathtools}
\usepackage{esint}
\usepackage[colorlinks=true, pdfstartview=FitV, linkcolor=blue, citecolor=blue, urlcolor=blue]{hyperref}
\usepackage[sort]{cite}
\numberwithin{equation}{section}

\newtheorem{lemma}{Lemma}
\newtheorem{theorem}{Theorem}[section]

\theoremstyle{definition}
\newtheorem{Ass}{Assumption}
\newtheorem{remark}{Remark}[section]

\newcommand{\xd}{\textrm{d}}
\newcommand{\R}{\mathbb{R}}

\newcommand{\diff}{\mathop{}\!\mathrm{d}}

\title[Rate of convergence of Yosida approximation for the nonlocal Cahn-Hilliard]{On the rate of convergence of Yosida approximation for the nonlocal Cahn-Hilliard equation}
\author[Piotr Gwiazda]{Piotr Gwiazda}
\address{{\it Piotr Gwiazda:} Institute of Mathematics, Polish Academy of Sciences, Warsaw, Poland}
\email{pgwiazda@mimuw.edu.pl}
\author[Jakub Skrzeczkowski]{Jakub Skrzeczkowski}
\address{{\it Jakub Skrzeczkowski:} Mathematical Institute, University of Oxford, Oxford, United Kingdom;  
Faculty of Mathematics, Informatics and Mechanics, University of Warsaw, Poland}
\email{jakub.skrzeczkowski@maths.ox.ac.uk}
\author[Lara Trussardi]{Lara Trussardi}
\address{{\it Lara Trussardi: }Department of Mathematics and Statistics, University of Konstanz, Universitätsstraße 10
78464 Konstanz, Germany -- Institute of Mathematics and Scientific Computing, University of Graz, Heinrichstraße 36, 8010 Graz (Austria)}
\email{lara.trussardi@uni-graz.at}

\keywords{Nonlocal Cahn-Hilliard equation, singular potential, Yosida approximation, Moreau-Yosida approximation, Hilbert-Schmidt integral operator, rate of convergence.}
\subjclass[2010]{45K05, 35K25, 35K55, 65M12.}
\begin{document}

\begin{abstract}
It is well-known that one can construct solutions to the nonlocal Cahn-Hilliard equation with singular potentials via Yosida approximation with parameter $\lambda 
\to 0$. The usual method is based on compactness arguments and does not provide any rate of convergence. Here, we fill the gap and we obtain an explicit convergence rate $\sqrt{\lambda}$. The proof is based on the theory of maximal monotone operators and an observation that the nonlocal operator is of Hilbert-Schmidt type. Our estimate can provide convergence result for the Galerkin methods where the parameter $\lambda$ could be linked to the discretization parameters, yielding appropriate error estimates. 
\end{abstract}

\maketitle

\tableofcontents

\section{Introduction}
\label{sec:intro}

The nonlocal version of the Cahn-Hilliard equation was introduced in the early $90$'s by G. Giacomin and J. Lebowitz~\cite{GL}. They considered the hydrodynamic limit of a microscopic model describing a $d$-dimensional lattice gas evolving via a Poisson nearest-neighbor process and derived a nonlocal energy functional of the form
\begin{align} \label{eq:energy}
E(u)&=\frac{1}{4}\int_{\Omega}\int_{\Omega}J(x-y)(u(x)-u(y))^{2} \diff x \diff y+ \int_{\Omega}F(u(x)) \diff x,
\end{align}
where $\Omega$ is a smooth bounded domain in $\R^d$, $J$ is a nonnegative and symmetric convolution kernel, and $F$ is a double-well potential.\\

The associated evolution problem is related to the gradient flow (in the $H^{-1}$-metric) and provides a nonlocal variant of the Cahn-Hilliard equation, given by the following system of equations:
\begin{align}\label{eq:NL}
\begin{aligned}
\partial_t u&-{\rm div}\,(m(u) \nabla \mu) =0\text{,}\\ 
\mu & =\frac{\delta E(u)}{\delta u}=a(x) u-J* u+ F'(u) \text{,}
\end{aligned}
\end{align}
equipped with appropriate boundary conditions, where $a(x)=(J*1)(x):=\int_{\Omega}J(x-y)\xd y$,  $(J*u)(x):=\int_\Omega J(x-y)u(y)\, \xd y$, for $x\in \Omega$. This equation describes the evolution of the concentration of two components in a binary fluid. The local concentration of one of the two components is represented by a real valued function $u = u(x)$, the function $\mu$ is the chemical potential, $F$ is a double-well potential, and $m$ is the mobility. In our case, we deal with the so-called nondegenerate Cahn-Hilliard equation, that is the mobility $m(u)$ is bounded away from $0$, and without loss of generality, we assume $m(u)=1$.  Moreover, the potential $F$ can be splitted into two parts
\begin{equation}\label{eq:splitting_potential_intro}
F(u) = \widehat{\gamma}(u) + \widehat{\Pi}(u), 
\end{equation}
where $\widehat{\gamma}$ is a convex and may be singular while $\widehat{\Pi}$ is a regular, small perturbation.\\

In~\cite{GL} the authors derived a free energy functional in nonlocal form, and proposed the corresponding gradient flow to model the phase-change in binary alloys. 
The mathematical literature on the nonlocal
Cahn-Hilliard equation is widely developed: among many others, we mention~\cite{ab-bos-grass-NLCH,bat-han-NLCH,gal-gior-grass-NLCH,gal-grass-NLCH,han-NLCH} and the references therein.
Concerning the existence of solution for the nonlocal Cahn-Hilliard equation, the first proof for singular kernels not falling within the $W^{1,1}$-existence theory and under possible degeneracy of the double-well potential was done in~\cite{DRST,DST}. These were also the first contributions dealing with the case of non-regular interaction kernels.
In~\cite{DST2} the authors provided a full characterization of existence in the case of singular double-well potential, and interaction kernels satisfying $W^{1,1}$ integrability assumptions. In~\cite{DRST,DST,DST2}, the authors worked under the assumption of constant mobility and consider a Yosida-type regularization on the nonlinearity. \\

As already mentioned, in this paper, we focus on the nondegenerate Cahn-Hilliard equation. This assumption is a mathematical simplification, allowing to study the Cahn-Hilliard equation in greater detail. In particular, there are much more results available for the nondegenerate equation, including well-posedness results \cite{MR4001523} or strict separation property \cite{poiatti20233d,GGGsep}. Nevertheless, it seems that it is the {\it degenerate} equation that is physically relevant as it results from several different limits, including hydrodynamic limit of Vlasov-type equation \cite{elbar-mason-perthame-skrzeczkowski}, interacting particle systems \cite{GL} or high-friction limit for the Euler-Korteweg equation \cite{gallenmuller2023cahn,elbar2023nonlocal,MR3615546}.  We also want to remark that there are few analytical results available for the degenerate equation \cite{MR3448925,MR4241616, elbar-skrzeczkowski, carrillo2023degenerate, elbar2021degenerate, MR4199231,elbar2023limit,carrillo2023competing} and that these studies were initiated by the paper of Elliott and Garcke \cite{MR1377481}.  In this paper we focus on the case with constant mobility since the nondegeneracy assumption is necessary to obtain an $L^2$ estimate on the chemical potential $\mu$ which is an essential part of our argument. \\

The standard methods of construction of solutions to \eqref{eq:NL} with singular potentials are based on compactness of solutions to the approximate system where the potential is regularized via Yosida approximation (see Section~\ref{sect:max_monotone_graphs} for the relevant definitions). Namely, the singular potential $\widehat{\gamma}$ is regularized via Yosida approximation $\widehat{\gamma_{\lambda}}$ and then one sends $\lambda \to 0$ based on a priori estimates. In this paper, we obtain an explicit estimate in $L^2((0,T)\times\Omega)$ which quantifies this convergence, see Theorem~\ref{thm:main} for more details. To our knowledge, this is the first result of this type. Such estimate could provide convergence of Galerkin methods for \eqref{eq:NL} when one applies Yosida approximation to regularize the potential and the Galerkin discretization in space. \\

We recall that the Yosida approximations of a function $F$ is defined as  $F_\lambda: \R \rightarrow \R$, having Lipschitz constant of order $1/\lambda$, and $\widehat{F}_\lambda(s) :=\int_0^s F_\lambda(r) \diff r$ for every $s \in\R$.
Our method consists in the following: we consider a sequence $\{u_{\lambda}\}$ satisfying 
\eqref{eq:approx_problem} given by
\begin{equation}\label{eq:approx_problem_intro}
\begin{split}
    \partial_t u_{\lambda} &= \Delta \mu_{\lambda},\mbox{ in }  \Omega\\
    \mu_{\lambda} &= B(u_{\lambda}) + \gamma_{\lambda}(u_{\lambda}) + \Pi(u_{\lambda}),\mbox{ in }  \Omega\\
    \frac{\partial \mu_{\lambda}}{\partial \textbf{n}} &= 0, \mbox{ on } \partial \Omega
\end{split}
\end{equation}
where $B(u_\lambda):=a(x)\,u_\lambda- J*u_\lambda$ and
with $\{\gamma_{\lambda}\}$ being a sequence of Yosida regularisations (see Section~\ref{sect:max_monotone_graphs} for the exact definition and properties). To obtain an explicit convergence rate we make two observations. The first one is the estimate on the Yosida approximation coming from \cite{CG00}:
\begin{equation}\label{eq:convergence_Yosida_general}
- \int_0^t \int_{\Omega} (\gamma_{\lambda_1}(u_{\lambda_1}) - \gamma_{\lambda_2}(u_{\lambda_2})) (u_{\lambda_1} - u_{\lambda_2}) \leq  C\, (\lambda_1 + \lambda_2).
\end{equation}
With \eqref{eq:convergence_Yosida_general}, it is easy to conclude the convergence with the rate $\sqrt{\lambda}$ in $H^{-1}(\Omega)$ norm as \eqref{eq:NL} is the gradient flow in $H^{-1}(\Omega)$ norm. To obtain the estimate in $L^2(\Omega)$, we use the nonlocal kernel $J$. We assume that $J \in L^2(\Omega)$ so that the operator $u \mapsto J\ast u$ is a Hilbert-Schmidt operator. Such operators are compact so that, up to a small error, their image is a finite-dimensional subspace of $L^2(\Omega)$ on which the two norms, $L^2(\Omega)$ and $H^{-1}(\Omega)$, are equivalent. This yields the information about the convergence in $L^2(\Omega)$. All these comments are made rigorous in Section \ref{sec:proof}. In particular, we state our result for the case with Neumann boundary conditions, but the theorem is valid also in the (simpler) case with Dirichlet boundary conditions or periodic
domains.\\

Several papers have addressed the discretization of \eqref{eq:NL} with the potential approximated via Yosida approximation \cite{MR3017623,MR2837798,MR3170501,MR3163106}. While we do not consider any discretization in our paper, one can apply several available results in the literature for the discretization of \eqref{eq:approx_problem_intro} with the regularized potential $\widehat{F}_{\lambda}$ (see \cite{MR3556407,MR4543432,MR3916957,MR4241519,MR3995984}) and then consider the limits: firstly $\Delta x, \Delta t \to 0$ and then $\lambda \to 0$. While the aforementioned results have been obtained for an explicitly given potential, we expect that they still hold true for an arbitrary regular potential.\\

One can even consider the joint limit $\Delta x, \Delta t, \lambda \to 0$: if $u^{\Delta x, \Delta t}_{\lambda}$ is a solution to the discretization of the regularized problem, we expect
$$
\|u^{\Delta x, \Delta t}_{\lambda} - u_{\lambda}\|_{L^2(0,T; L^2(\Omega))} \leq C(\lambda) \, \left[|\Delta t|^{\alpha} + |\Delta x|^{\beta}\right],
$$
where $\alpha$, $\beta$ depends on the particular method and $C(\lambda) \to \infty$ when $\lambda \to 0$. Then, applying Theorem \ref{thm:main} one gets
$$
\|u^{\Delta x, \Delta t}_{\lambda} - u\|_{L^2(0,T; L^2(\Omega))}  \leq C(\lambda) \, \left[|\Delta t|^{\alpha} + |\Delta x|^{\beta}\right] + C\, \sqrt{\lambda}
$$
so that if $\Delta t$, $\Delta x$ are chosen to depend (appropriately) on $\lambda$, one finally obtains the convergence 
$$
\|u^{\Delta x, \Delta t}_{\lambda} - u\|_{L^2(0,T; L^2(\Omega))} \to 0.
$$
The details depend on the rates $\alpha$, $\beta$ as well as the constant $C(\lambda)$ (i.e. how much the convergence result depends on the regularity of the potential). While the rates $\alpha$, $\beta$ are known for several methods, the constant $C(\lambda)$ has to be obtained by a careful inspection of the convergence proofs and we leave it for future works. We also remark that other schemes have been proposed for the Cahn-Hilliard equation and some related problems in \cite{MR1742748, GLWW,GWW,BCK2023,agosti2023strict,elbar2023analysis,BMS14,CX19,GLW}. \\

The paper is structured as follows: in Section~\ref{sec:setting} we introduce the setting and recall the main tools used in the paper; in Section~\ref{sec:estimate} we compute the uniform estimates and in Section~\ref{sec:proof} we provide the proof of the main theorem (Theorem \ref{thm:main}). 

\section{Assumptions and the main result}
\label{sec:setting}

In order to state our result we need several definitions and tools that we are going to recollect in this section.  In particular, we introduce the notations and hypothesis used in this paper and recall the maximal monotone graphs theory as well as some basic facts on the Hilbert-Schmidt integral operator.

\subsection{Hypotheses}
\begin{Ass}\label{ass:the_main_ass}
Throughout the paper
we assume the following:
\begin{enumerate}[label=(H\arabic*)]
\item $\Omega$ is a smooth bounded domain in $\R^d$, and $T > 0$ is a fixed final time.
\item \label{ass:H2} The kernel $J:
\R^d \to \R^+$ satisfies $J\in W^{1,1}_{\text{loc}}(\R^d)\cap L^2_{\text{loc}}(\R^d)$ is such that $J(x) = J(-x)$ for almost every $x$. 
For any measurable $v:\Omega\rightarrow\mathbb{R}$ we use the notation 
$$
(J*v)(x):=\int_\Omega J(x-y)\, v(y) \diff y,\quad x\in\Omega
$$
and set $a(x)=J*1$. We also assume that for some $a_->0$
\begin{equation}\label{eq:ass_a}
\inf_{x \in \Omega} a(x) > a_{-} > 0.
\end{equation}
\item \label{ass:H3}We assume that $F = \widehat\gamma+\widehat\Pi$, where $\widehat \gamma$ is a convex proper function. We write  $\gamma:\mathbb{R}\to2^\mathbb{R}$ for a maximal
monotone graph such that $0\in\gamma(0)$ and $\gamma$ is a subdifferential of $\widehat\gamma$ in the convex analysis sense. Moreover, we let $\Pi = \widehat{\Pi}'$ and we assume that $\Pi$ is a Lipschitz
function such that
\begin{equation}\label{eq:ass_Lip_Pi}
a_{-} - \|\Pi'\|_{\infty} > 0.
\end{equation}
Without loss of generality we can suppose that $F$ is nonnegative.

\item \label{ass:H4}The initial condition $u_0:\Omega\to \R$ is such that $u_0 \in L^2(\Omega)$ and $\widehat{\gamma}(u_0)\in L^1(\Omega)$. Moreover, we assume that the average of $u_0$
\begin{equation}\label{eq:condition_average_init_cond}
\frac{1}{|\Omega|} \int_{\Omega} u_0 \diff x \in (m_{-}, m_+)
\end{equation}
with $m_- <0< m_+$ and $m_-, m_+\in \mbox{dom}(\gamma)$, where the domain $\mbox{dom}(\gamma)$ is defined in Section \ref{sect:max_monotone_graphs}.
\end{enumerate}
\end{Ass}

\phantom{...}\\
Let us comment the assumptions. Concerning \ref{ass:H2}, the kernel $J$ is required to be defined on $\R^d$ but in fact, it could be defined on the set $\Omega - \Omega$, i.e. the set containing all the points of the form $x-y$ where $x,y \in \Omega$. The $L^2$ regularity of the kernel is necessary to know that the nonlocal operator $J \ast v$ is an Hilbert-Schmidt operator from $L^2(\Omega)$ to $L^2(\Omega)$. We remark that condition \ref{ass:H2} implies 
$$
\nabla a(x) \in L^{\infty}(\Omega).
$$
Moreover, the lower bound on $a$ is a technical matter related to the geometry of $\Omega$ and symmetry of $J$: it is easy to see that it holds true if $\Omega$ is a ball and $J$ is radially symmetric. Concerning~\ref{ass:H3}, the assumption on the potential is fairly standard: the main part $\widehat{\gamma}$ is convex and possibly singular while the term $\widehat{\Pi}$ should be considered as a perturbation which cannot be too large. The reason for considering $\gamma$ as a maximal monotone graph is that the derivative of $\widehat{\gamma}$ may not exist in the classical sense but it exists as a subdifferential of a convex function which is a maximal monotone graph, see Section \ref{sect:max_monotone_graphs} for the relevant definitions. Finally, the first part of \ref{ass:H4} guarantees that the energy related to \eqref{eq:approx_problem_intro} is finite at $t=0$. The condition~\eqref{eq:condition_average_init_cond} says that the initial condition is not supported only at the points where subdifferential $\gamma$ does not exist.   \\

Hence, the most demanding condition is \ref{ass:H3} which still allows us to consider all reasonable choices for the potential. The standard choice for $F$ is the fourth-order polynomial $F_{\text{pol}}(x):=\frac{1}{4}(x^2-1)^2, x\in\R$, with minima, corresponding to the pure phases, in $\pm 1$. A more realistic description is given by the logarithmic double-well potential
\begin{equation}\label{eq:logarithmic_F}
F_{\text{log}}(x):=\frac{\theta}{2}[(1+x)\log(1+x)+(1-x)\log(1-x)]+\frac{\Theta}{2}-cx^2
\end{equation}
with $0<\theta<\Theta$ and $c>0$. With this choice then we are on the bounded domain $(-1, 1)$ and
the minima are within the open interval $(-1, 1)$.
A third example which is also included in our study is the so-called double-obstacle potential (see~\cite{MR1123143,oono1988study})
\begin{equation}\label{eq:example_F_singular}
F_{\text{ob}}(x):=I_{[-1,1]}(x)+\frac{1}{2}(1-x^2),  \quad I_{[-1,1]}(x):=
\begin{cases}
0&\text{if }x\in [-1,1]\\+\infty&\text{otherwise}\,.
\end{cases} .
\end{equation}
With this choice $F'_{\text{ob}}$ is not defined in the usual way, and has to be interpreted as the subdifferential $\partial F_{\text{ob}}$ in the sense of convex analysis.\\

The necessity of singular potentials as in \eqref{eq:splitting_potential_intro} is motivated by one of the first derivation of the (nonlocal) Cahn-Hilliard equation due to Giacomin-Lebowitz \cite{GL,MR1638739} who considered a logarithmic potential \eqref{eq:logarithmic_F}. Furthermore, the double-obstacle potential \eqref{eq:example_F_singular} was proposed by Oono and Puri \cite{oono1988study} to model phase separation (see also \cite{MR1166255, MR1123143}). One of its interesting application is in the inpainting of damaged images \cite{MR3148076,MR3852712} where the double-obstacle potential leads to better visual results comparing to the smooth potentials. Last but not least, the singular potentials appear in several applications, including tumor growth \cite{MR3636313,MR4502226,MR4537582} and flows of the binary mixtures \cite{MR4175449,MR4354132}. Our work covers also several classes of singular kernels $J$, including Riesz, Newtonian, and Bessel, which are used to model nonlocal interactions in multiple settings, including tumor growth \cite{MR4188329, MR3393324}, (Patlak-)Keller-Segel and aggregation-diffusion equation \cite{MR2013508,MR2914242,MR4022083} and related applications in the sampling problems \cite{MR4582478}. Nevertheless, we point out that the focus of the current paper is on the singular potential rather than on the singular kernel.\\

We also highlight that our result could be generalized to anisotropic potential possibly dependent on time, i.e. of type $F=F(u,x,t)$ but additional hypothesis should be specified. \\

\subsection{Maximal monotone graphs theory}\label{sect:max_monotone_graphs}

As specified in \ref{ass:H3} we assume that $F$ can be written as 
$$
F = \widehat{\gamma} + \widehat{\Pi},
$$
where $\widehat{\gamma}: \R \to \R \cup \{\infty\}$ is a convex, lower semicontinuous and $\widehat{\Pi}: \R \to \R$ is a function such that $\Pi:= \widehat{\Pi}'$ is globally Lipschitz continuous. For such $\widehat{\gamma}$, we define its domain as $D(\widehat{\gamma}) = \{x \in \R: \widehat{\gamma}(x) < \infty \}$.\\

The subdifferential ${\gamma}(x)$ is defined as a set-valued map satisfying 
$$
\widehat{\gamma}(y) \geq \widehat{\gamma}(x) + \gamma(x) \, (y-x). 
$$
Here, we slightly abuse the notation and we highlight that we write $\gamma(x)$ for an arbitrary element of the set $\gamma(x)$. It may happen that $\gamma(x)$ is empty for some $x$; hence, we define its domain which is the set of points where $\widehat{\gamma}$ is differentiable $\mbox{dom}(\gamma) = \{x: \gamma(x) \neq \emptyset \}$. The subdifferential $\gamma$ is a maximal monotone graph due to the celebrated result of Rockafellar \cite[Corollary 31.5.2]{MR0274683}. This means that $\gamma(x)$ is monotone
$$
(\gamma(x) - \gamma(y)) \, (x-y) \geq 0 \qquad \qquad \forall{x,y \in \mbox{dom}(\gamma)}
$$
and that there is no bigger (in the sense of inclusions of graphs) multi-valued map which is monotone. Below we briefly recall the most important facts while for the complete theory we refer to \cite[Chapter 3]{AC84} and \cite{brezis}. \\

Concerning the relation between $D(\widehat{\gamma})$ and $\mbox{dom}(\gamma)$, it is well-known that if $\widehat{\gamma}(x)<\infty$ and $
\widehat{\gamma}$ is continuous at $x$, then $
\gamma(x)$ is nonempty cf. \cite[Proposition 3.29 (ii)]{MR3821514}. The continuity assumption is not restrictive since any convex function is continuous on the interior of its effective domain, see \cite[Proposition 2.20]{MR3821514}. On the other hand, if $\gamma(x)$ is nonempty, it means that $x \in D(\widehat{\gamma})$ (otherwise, $\gamma(x)$ has to be empty). Therefore, $\mbox{dom}(\gamma) \subset D(\widehat{\gamma})$ and this inclusion can be strict, see \cite[p.218]{MR0274683}. We note the following:

\begin{remark}\label{rem:all_the_interval_belongs}
If $a, b \in \mbox{dom}(\gamma)$, then $(a,b) \subset \mbox{dom}(\gamma)$. Indeed $\widehat{\gamma}$ has to be finite on $[a,b
]$ and continuous on $(a,b)$. 
\end{remark}

To introduce Yosida approximation, it is important to recall that the fact that $\gamma$ is a maximal monotone graph is equivalent with saying that $( I + \lambda \gamma)^{-1}$ is a contraction for all $\lambda > 0$ well-defined on $\R$ cf. \cite[Proposition 2.2]{MR0348562}. It means that
$$
|( I + \lambda\, \gamma)^{-1}(x_1) - ( I + \lambda\, \gamma)^{-1}(x_2)| \leq  |x_1 - x_2|
$$
for all $x_1, x_2 \in \R$. More precisely, monotonicity always implies that $( I + \lambda \gamma)^{-1}$ is a contraction but it is the maximal monotonicity which implies that the range of $I + \lambda \gamma$ is the whole space $\R$ so that $( I + \lambda \gamma)^{-1}$ is defined on $\R$.  \\

\noindent For the sake of analysis, we need to approximate multi-valued map $\gamma$ with a single-valued map via the so-called Yosida approximation. To this end, we let
\begin{equation}\label{eq:def_gamma_lambda}
\gamma_{\lambda}:= \frac{I - J_{\lambda}}{\lambda}, \qquad \qquad J_{\lambda} = (I + \lambda \, \gamma)^{-1},
\end{equation}
where $I$ is the identity map. Then, the Yosida approximation is defined with
\begin{equation}\label{eq:Yosida_approx}
\widehat{\gamma_{\lambda}}(x) = \min_{y \in \R} \left\{\frac{1}{2\lambda} |y-x|^2 + \widehat{\gamma}(y) \right\}
\end{equation}
Formula \eqref{eq:Yosida_approx} immediately implies that $\widehat{\gamma_{\lambda}}(x)$ is convex.\\

To understand better the relation between $\widehat{\gamma_{\lambda}}$ and ${\gamma_{\lambda}}$, we note that the minimum has to be attained for $y$ such that $\frac{x-y}{\lambda} \in \partial \widehat{\gamma}(y) = \gamma(y)$ which means that $y = J_{\lambda}(x)$ and the formula simplifies to
$$
\widehat{\gamma_{\lambda}}(x) = \frac{\lambda}{2} |\gamma_{\lambda}x|^2 + \widehat{\gamma}(J_{\lambda}x).
$$
We also note that from this reasoning we have $\frac{x-J_{\lambda}(x)}{\lambda} \in \gamma(J_{\lambda}(x))$ so that
\begin{equation}\label{eq:fact_about_gamma_lambda_in_the_resolvent}
\gamma_{\lambda}(x) \in \gamma(J_{\lambda}(x)),
\end{equation}
which will be useful in our reasoning. Finally, one can check that $\widehat{\gamma_{\lambda}}$ is classically differentiable, $\widehat{\gamma_{\lambda}}' = \gamma_{\lambda}$ and $\widehat{\gamma_{\lambda}} \nearrow \widehat{\gamma}$ as $\lambda \to 0$ cf. \cite[Proposition 2.11]{MR0348562}.\\

\noindent Concerning the behaviour of the derivatives $\gamma_{\lambda}$, we first define $\gamma_0(x)$ to be an element of $\gamma(x)$ with a minimal norm (it exists by maximal monotonicity). Then, it is known cf. \cite[Proposition 2.6]{MR0348562} that
\begin{equation}\label{eq:convergence_to_the_minimal_element}
|\gamma_{\lambda}(x)| \nearrow |\gamma_0(x)| \mbox{ for all } x\in \mbox{dom}(\gamma).
\end{equation}

We conclude this discussion with a consequence of the assumption $0 \in \gamma(0)$. Namely, we have \begin{equation}\label{eq:gamma_lambda_at_0}
\gamma_{\lambda}(0) = 0.
\end{equation}
Indeed, since $0 \in \gamma(0)$, the unique solution to equation $0 = u + \lambda\, \gamma(u)$ equals $u=0$ so that $J_{\lambda}(0) = 0$. Using \eqref{eq:def_gamma_lambda}, we obtain $\gamma_{\lambda}(0) = 0$.
\subsection{The nonlocal operator}
Let us consider the operator
\[
Bu(x):=\int_\Omega J(x - y)(u(x)-u(y))\diff y.
\]
First, it is worth considering only a part of $B$, namely 
$$
Iu(x):= \int_\Omega J(x - y) \, u(y)\diff y.
$$
It follows that since $J \in L^2_{\text{loc}}(\R^d)$, the operator $I: L^2(\Omega) \to L^2(\Omega)$ is Hilbert-Schmidt~\cite[Theorem 8.83]{RR04}. In particular, it is a compact operator (that is, the image of the unit ball is relatively compact) and for any orthonormal basis $\{e_i\}$ in $L^2(\Omega)$, we have (see \cite[Appendix C]{Da_Prato_2014}):
\begin{equation}
\sum_{i=1}^{\infty} \|I e_i\|_{L^2(\Omega)}^2 <\infty.
\end{equation}
Concerning the operator $B$, we have the following properties which are a simple consequence of the symmetry and the nonnegativity of $J$:
\begin{lemma}\label{lem:3props_on_nonloca}
    It holds:
\begin{enumerate}[label=(\roman*)] 
    \item\label{prop1} $\int_{\Omega} Bu\, v \diff x = \int_{\Omega} Bv\, u \diff x$,
    \item\label{prop2} $\int_{\Omega} Bu \diff x = 0$,
    \item\label{prop3} $\int_{\Omega} Bu\, u \diff x = \int_{\Omega} \int_{\Omega} J(x-y)(u(x) - u(y))^2 \diff x \diff y \geq 0$.
    \end{enumerate}
\end{lemma}
\begin{proof}
    We first note that since $J$ is symmetric
$$
\int_{\Omega} \int_{\Omega} J(x-y)\,(u(x)-u(y))\,v(x) \diff x \diff y = 
-\int_{\Omega} \int_{\Omega} J(x-y)\,(u(x)-u(y))\,v(y) \diff x \diff y 
$$
which implies
$$
\int_{\Omega} Bu\, v \diff x = \frac{1}{2} \int_{\Omega} \int_{\Omega} J(x-y)\,(u(x)-u(y))\,(v(x)-v(y)) \diff x \diff y.
$$
Hence, \ref{prop3} follows immediately. For \ref{prop1}, we observe that the expression above does not change when $u$ and $v$ are interchanged. Finally, \ref{prop2} follows from \ref{prop1} because $B(1) = 0$.
\end{proof}
\medskip

\subsection{The main result}
The nonlocal Cahn-Hilliard equation we are going to consider has then the following form:
\begin{equation}\label{eq:nlch}
\begin{split}
    \partial_t u &= \Delta \mu,\mbox{ in }  \Omega\\
    \mu &= B(u) + \gamma(u) + \Pi(u),\mbox{ in } \Omega
\\
    \frac{\partial \mu}{\partial \textbf{n}} &= 0, \mbox{ on } \partial \Omega.
\end{split}
\end{equation}

Using Yosida approximation defined in Section \ref{sect:max_monotone_graphs}, we consider solutions to 
\begin{equation}\label{eq:approx_problem}
\begin{split}
    \partial_t u_{\lambda} &= \Delta \mu_{\lambda},\mbox{ in }  \Omega\\
    \mu_{\lambda} &= B(u_{\lambda}) + \gamma_{\lambda}(u_{\lambda}) + \Pi(u_{\lambda}),\mbox{ in }  \Omega\\
    \frac{\partial \mu_{\lambda}}{\partial \textbf{n}} &= 0, \mbox{ on } \partial \Omega
\end{split}
\end{equation}
where $\gamma_{\lambda}$ is defined in \eqref{eq:def_gamma_lambda}. Our main results reads:
\begin{theorem}\label{thm:main}
Suppose that Assumption \ref{ass:the_main_ass} holds true. Then,
$$
\|u_{\lambda} - u\|_{L^2((0,T)\times\Omega)} \leq C\,\sqrt{\lambda},
$$
where the constant $C$ depends on the model functions $\widehat{\Pi}$, $\widehat{\gamma}$ and the norm of the initial condition in $L^2(\Omega)$.
\end{theorem}
It will be clear from the proof that the constant depends mostly on the distance between $a_{-}$ and $\|\Pi'\|_{\infty}$ as in \eqref{eq:ass_Lip_Pi}.\\



\section{Uniform estimates}
\label{sec:estimate}

We recall that existence of solution to the nonlocal Cahn-Hilliard equation with $W^{1,1}$ kernel has been proved in~\cite{DST2}.
In this section, we consider solutions to \eqref{eq:approx_problem} and prove the following uniform estimates:
\begin{theorem}\label{thm:uniform_estimates}
The following sequences are bounded:
\begin{enumerate}[label=(\Alph*)]   
\item\label{estim1} $\{u_{\lambda}\}$ in $L^{\infty}_t L^2_x$,
\item\label{estim2} $\{\nabla u_{\lambda}\}$ in $L^{2}_t L^2_x$
\item\label{estim3} $\{ \mu_{\lambda}\}$ in $L^{2}_t L^2_x$,
\item\label{estim4} $\{\nabla \mu_{\lambda}\}$ in $L^{2}_t L^2_x$,
\item\label{estim5} $\{\gamma_{\lambda}(u_{\lambda})\}$ in $L^2_t L^2_x$.
\end{enumerate}
\end{theorem}
The proof of Theorem \ref{thm:uniform_estimates} is split into several parts. 
\begin{lemma}\label{lem:part1_of_unif_estimates}
Assertions \ref{estim1}, \ref{estim2} and \ref{estim4} in Theorem \ref{thm:uniform_estimates} hold true.     
\end{lemma}
\begin{proof}
We multiply \eqref{eq:approx_problem} with $u_{\lambda}$ to get
\begin{equation}\label{eq:estimates_mult_by_u}
\begin{split}
\partial_t \int_{\Omega} \frac{u_{\lambda}^2}{2} = \int_{\Omega} \Delta \mu_{\lambda} \, u_\lambda &= - \int_{\Omega} \nabla \mu_{\lambda} \, \nabla u_{\lambda} = \\ &=
-\int_{\Omega} \Pi'(u_\lambda) |\nabla u_{\lambda}|^2 - \int_{\Omega} \gamma_{\lambda}'(u_{\lambda}) |\nabla u_{\lambda}|^2 - \int_{\Omega} \nabla B(u_{\lambda}) \cdot \nabla u_{\lambda}.
\end{split}
\end{equation}
Now, by convexity of the Yosida approximation, we have 
\begin{equation}\label{eq:estimates_gamma'_geq0}
\int_{\Omega} \gamma_{\lambda}'(u_{\lambda}) |\nabla u_{\lambda}|^2 \geq 0
\end{equation}
(in fact, it may happen that $\gamma_{\lambda}'$ does not exist, see Remark \ref{rem:second_derivative_gamma} below). Concerning the term with the bilinear form, we note that $B(u_{\lambda}) = a(x) \, u_{\lambda} - J\ast u_{\lambda}$ so that
\begin{equation}\label{eq:estimates_bilinear_expanded}
\int_{\Omega} \nabla B(u_{\lambda}) \cdot \nabla u_{\lambda} =  \int_{\Omega} a(x)\, |\nabla u_{\lambda}|^2 + \int_{\Omega} \nabla a(x)\, \nabla u_{\lambda}(x)\,u_{\lambda}(x) - \int_{\Omega} (\nabla J \ast u_{\lambda}) \nabla u_{\lambda}.
\end{equation}
The last two terms can be estimated via Cauchy-Schwarz and Young convolutional inequalities:
\begin{equation}
\left|\int_{\Omega} \nabla a(x)\, \nabla u_{\lambda}(x)\,u_{\lambda}(x) \right| \leq \varepsilon\, \int_{\Omega} |\nabla u_{\lambda}|^2 + C(\varepsilon) \, \|\nabla a\|_{\infty}^2  \int_{\Omega} |u_{\lambda}|^2   
\end{equation}
\begin{equation}\label{eq:estimates_convolution_term_Young}
\left|\int_{\Omega} (\nabla J \ast u_{\lambda}) \nabla u_{\lambda} \right| \leq
\varepsilon \int_{\Omega} |\nabla u_{\lambda}|^2 + C(\varepsilon) \left(\int_{\Omega} |\nabla J|\right)^2 \int_{\Omega} |u_{\lambda}|^2 
\end{equation}
where $\varepsilon$ has to be chosen. Collecting \eqref{eq:estimates_mult_by_u}--\eqref{eq:estimates_convolution_term_Young} and recalling that $a(x) > a_{-}$ (see \eqref{eq:ass_a}) we conclude 
$$
\partial_t \int_{\Omega} \frac{u_{\lambda}^2}{2} + (a_{-} - \| \Pi' \|_{\infty} - \varepsilon) \int_{\Omega} |\nabla u_{\lambda}|^2 \leq C(\varepsilon, \nabla J, \nabla a)\, \int_{\Omega} |u_{\lambda}|^2.
$$
We choose $\varepsilon>0$ sufficiently small such that
$$
a_{-} - \| \Pi' \|_{\infty} - \varepsilon > 0
$$
which is possible thanks to \eqref{eq:ass_Lip_Pi}. Since $u_0 \in L^2(\Omega)$, this concludes the proofs of \ref{estim1} and \ref{estim2}. \\

\noindent Now, we multiply \eqref{eq:approx_problem} with $\mu_\lambda$ and integrate in space using Neumann's condition on $\mu_{\lambda}$ to obtain
\begin{equation}\label{eq:energy_for_lambda}
\partial_t \int_{\Omega} \widehat{F}_\lambda(u_{\lambda}) + 
\int_{\Omega} B(u_{\lambda}) \partial_t u_{\lambda} + \int_{\Omega} |\nabla \mu_{\lambda}|^2 = 0
\end{equation}
where $\widehat{F}_\lambda$ is a primitive function of $F_{\lambda}$. Recall that $B(u_\lambda) = u_{\lambda} - J\ast u_{\lambda}$ so that the symmetry of the kernel $J$ implies
$$
\int_{\Omega} B(u_{\lambda}) \partial_t u_{\lambda} = \frac{1}{4} \partial_t \int_\Omega \int_\Omega J(x-y) (u_\lambda(x) - u_\lambda(y))^2 \diff x \diff y.
$$
Therefore, estimate \ref{estim4} follows from \eqref{eq:energy_for_lambda} after integrating in time and using assumption on the initial condition.
\end{proof}
\begin{lemma}\label{lem:L1_est_on_potential}
Suppose that $\{u_{\lambda}\}$ is bounded in $L^{\infty}_t L^2_x$. Then, the sequence $\{\gamma_{\lambda}(u_{\lambda})\}$ is uniformly bounded in $L^2_t L^1_x$.
\end{lemma}
\begin{proof}
Given $f \in L^2(\Omega)$ with $\int_{\Omega} f = 0$, there exists unique $w \in H^1(\Omega)$ such that
\begin{equation}\label{eq:Laplace_operator}
-\Delta w = f \mbox{ on } \Omega, \qquad \frac{\partial}{\partial {\bf n}} w = 0.
\end{equation} 
We write $w = (-\Delta)^{-1} f$. The operator $(-\Delta)^{-1}$ is nonnegative and self-adjoint. Therefore, it has a self-adjoint square root that we denote with $(-\Delta)^{-1/2}$.\\

\noindent We multiply \eqref{eq:approx_problem} with $(-\Delta)^{-1}(u_{\lambda} - \overline{u_{0}})$ where $\overline{u_0}$ is the average of $u_0$ to get 
\begin{equation}\label{eq:without_integrating_Delta}
\langle \partial_tu_{\lambda}, (-\Delta)^{-1}(u_{\lambda} - \overline{u_{0}})\rangle_{(H^{-1}, H^1)} + \int_{\Omega} \mu_{\lambda} (u_{\lambda} - \overline{u_{0}}) = 0.
\end{equation}
Thanks to \ref{estim4} and equation \eqref{eq:approx_problem}, $\{\partial_tu_{\lambda}\}$ is uniformly bounded in $L^2_t H^{-1}_x$. Moreover, thanks to \ref{estim1}, $ (-\Delta)^{-1}(u_{\lambda} - \overline{u_{0}})$ is uniformly bounded in $L^{\infty}_t H^1_x$. It follows that the expression $\langle \partial_tu_{\lambda}, (-\Delta)^{-1}(u_{\lambda} - \overline{u_{0}})\rangle_{(H^{-1}, H^1)}$ is uniformly bounded in $L^2(0,T)$. The last integral in \eqref{eq:without_integrating_Delta} can be split into three terms:
$$
\int_{\Omega} \mu_{\lambda} (u_{\lambda} - \overline{u_{0}}) = \int_{\Omega} B(u_{\lambda}) \, (u_{\lambda} - \overline{u_{0}})
+ 
\int_{\Omega} \Pi(u_{\lambda})\, (u_{\lambda} - \overline{u_{0}})
+
\int_{\Omega} \gamma_{\lambda}(u_{\lambda}) (u_{\lambda} - \overline{u_{0}})
$$
which can be studied separately. Clearly,
$$
\int_{\Omega} B(u_{\lambda}) \, (u_{\lambda} - \overline{u_{0}}) = \frac{1}{2} \,\int_{\Omega} \int_{\Omega} J(x-y) |u_{\lambda}(x) - u_{\lambda}(y)|^2 \diff y \diff x \geq 0
$$
because $\int_{\Omega} B(u_{\lambda}) \, \overline{u_0} = 0$. For the term with $\Pi(u_{\lambda})$ we estimate
\begin{equation}\label{eq:Pi_Lipschitz_estimate}
\left|\Pi(u_{\lambda}) \right| \leq | \Pi(u_{\lambda}) - \Pi(0) | + |\Pi(0)| \leq \|\Pi'\|_{\infty} |u_{\lambda}| + |\Pi(0)|,
\end{equation}
so that $\int_{\Omega} \Pi(u_{\lambda})\, (u_{\lambda} - \overline{u_{0}})$ is bounded in $L^{\infty}(0,T)$ due to \ref{estim1}. Finally, we have to estimate the term with $\gamma_{\lambda}(u_{\lambda})$. For this, we prove that there exist constants $M_1$, $M_2$ depending only on $\overline{u_{0}}$ such that
\begin{equation}\label{eq:estimate_modulus_gamma}
M_1 |\gamma_{\lambda}(r)| + M_2 \leq \gamma_{\lambda}(r) (r - \overline{u_{0}}).
\end{equation}
If \eqref{eq:estimate_modulus_gamma} is proved, then the proof of Lemma \ref{lem:L1_est_on_potential} is completed because the rest of the terms are either nonnegative or bounded in $L^2(0,T)$. \\

\noindent In order to prove~\eqref{eq:estimate_modulus_gamma}, we follow the argument from \cite[p. 908]{gil-mir-sch}. We need the assumption that there exists $m_- < 0 < m_+$ such that $m_-, m_+ \in \mbox{dom}(\gamma)$ and $\overline{u_{0}} \in (m_-, m_+)$ as in \eqref{eq:condition_average_init_cond}. We define 
$$
\delta_0 := \min(\overline{u_{0}}- m_-, m_+ - \overline{u_{0}}).
$$

\noindent \underline{\emph{Case $r \geq m_+$ or $r \leq m_-$.}} We only consider $r \geq m_+$ as the second case is similar. Due to \eqref{eq:gamma_lambda_at_0}, we have $\gamma_{\lambda}(0) = 0$. As $\gamma_{\lambda} = ( \widehat{\gamma_{\lambda}})'$ is nondecreasing (this follows from convexity of $\widehat{\gamma_{\lambda}}$, cf. \eqref{eq:Yosida_approx}), $\gamma_{\lambda}(r) \geq \gamma_{\lambda}(0) = 0$. Since $r-\overline{u_{0}} \geq \delta_0$, estimate \eqref{eq:estimate_modulus_gamma} is satisfied with $M_1=\delta_0$ and any nonpositive $M_2$.\\

\noindent \underline{\emph{Case $r \in (m_-, m_+)$.}} Note that, by Remark \ref{rem:all_the_interval_belongs}, the interval $[m_-, m_+]$ belongs to $\mbox{dom}(\gamma)$. Since $-\gamma_{\lambda}(r)\, (r-\overline{u_0}) \leq |\gamma_{\lambda}(r)|\, (m_+ - m_-)$ we can simply estimate
$$
\delta_0 \, |\gamma_{\lambda}(r)| - \gamma_{\lambda}(r) (r - \overline{u_{0}}) \leq (\delta_0 + m_+ - m_-)\, |\gamma_{\lambda}(r)| \leq (\delta_0 + m_+ - m_-)\, \sup_{r \in [m_-, m_+]} |\gamma_0(r)|,
$$
where we used monotone convergence of Yosida approximations to the minimal value $\gamma_0$, see \eqref{eq:convergence_to_the_minimal_element}. Of course, by monotonicity, the supremum above can be estimated only in terms of $|\gamma_0(m_-)|$ and $|\gamma_0(m_+)|$, so the proof is concluded.   
\end{proof}

\begin{proof}[Proof of Theorem \ref{thm:uniform_estimates}]

\noindent It remains to establish \ref{estim3} and \ref{estim5}. Concerning \ref{estim3}, by Poincaré inequality, it is sufficient to prove that the average $\int_{\Omega} \mu_{\lambda}$ is bounded in $L^2(0,T)$. By definition of $\mu_{\lambda}$, we have
$$
\int_{\Omega} \mu_{\lambda} = \int_{\Omega} B(u_{\lambda}) + \int_{\Omega} \gamma_{\lambda}(u_{\lambda}) + \int_{\Omega} \Pi(u_{\lambda}). 
$$
Note that $\int_{\Omega} B(u_{\lambda})  = 0$ by \ref{prop2} in Lemma \ref{lem:3props_on_nonloca}, $\int_{\Omega} \gamma_{\lambda}(u_{\lambda})$ is bounded in $L^2(0,T)$ by Lemma \ref{lem:L1_est_on_potential} while for the last term we use \eqref{eq:Pi_Lipschitz_estimate} and apply estimate \ref{estim1}. The conclusion follows.\\

\noindent Finally, we prove \ref{estim5}. For this, using the formula for $\mu_{\lambda}$, we get:
\begin{equation}\label{eq:formula_gamma_lambda}
\gamma_{\lambda}(u_{\lambda}) = \mu_{\lambda} - (J\ast1) u_{\lambda} + J\ast u_{\lambda} - \Pi(u_{\lambda}).
\end{equation}
We want to prove that each term appearing on the right-hand side (RHS) of \eqref{eq:formula_gamma_lambda} is bounded in $L^2_t L^2_x$. For $\mu_{\lambda}$, this follows from \ref{estim3} while for $(J\ast1) u_{\lambda}$, this follows from \ref{estim1}. For the term $J\ast u_{\lambda}$ we apply Young's convolutional inequality (note that $J$ does not depend on time) to deduce
$$
\|J \ast u_{\lambda}\|_{L^2_t L^2_x} \leq \|J\|_{L^1_x} \|u_{\lambda}\|_{L^2_t L^2_x}
$$
which is bounded by \ref{estim1}. Finally, for $\Pi(u_{\lambda})$ we apply \eqref{eq:Pi_Lipschitz_estimate} and \ref{estim1} once again.
\end{proof}

\begin{remark}\label{rem:second_derivative_gamma}
In the proof of Lemma \ref{lem:part1_of_unif_estimates}, we used the existence of $\gamma_{\lambda}'$ which may not be the case (we only know that $\gamma_{\lambda}$ is continuous for fixed $\lambda>0$). However, in fact, we only used the sign which can be deduced from convexity by a suitable approximation scheme. Namely, let $\gamma_{\lambda}^{\varepsilon}$ be a usual mollification of $\gamma_{\lambda}$ and let $u^{\varepsilon}_{\lambda}$ be solution to
\begin{equation}\label{eq:approx_problem_eps}
\begin{split}
    \partial_t u^{\varepsilon}_{\lambda} &= \Delta \mu^{\varepsilon}_{\lambda},\mbox{ in }  \Omega\\
    \mu^{\varepsilon}_{\lambda} &= B(u^{\varepsilon}_{\lambda}) + \gamma^{\varepsilon}_{\lambda}(u^{\varepsilon}_{\lambda}) + \Pi(u^{\varepsilon}_{\lambda}),\mbox{ in }  \Omega\\
    \frac{\partial \mu^{\varepsilon}_{\lambda}}{\partial \textbf{n}} &= 0, \mbox{ on } \partial \Omega.
\end{split}
\end{equation}
Note that since $\gamma_{\lambda}$ is nondecreasing, the same holds for $\gamma^{\varepsilon}_{\lambda}$ so that $\gamma^{\varepsilon}_{\lambda}\,' \geq 0$. Arguing as in \eqref{eq:estimates_mult_by_u}--\eqref{eq:estimates_gamma'_geq0}, we obtain the inequality
$$
\partial_t \int_{\Omega} \left(u^{\varepsilon}_{\lambda}\right)^2 \leq 
-\int_{\Omega} \Pi'(u^{\varepsilon}_\lambda) |\nabla u^{\varepsilon}_{\lambda}|^2  - \int_{\Omega} \nabla B(u^{\varepsilon}_{\lambda}) \cdot \nabla u^{\varepsilon}_{\lambda}.
$$
As in Lemma \ref{lem:part1_of_unif_estimates}, we get the following uniform bounds: $\{u^{\varepsilon}_{\lambda}\}$ in $L^{\infty}_t L^2_x$, $\{\nabla u^{\varepsilon}_{\lambda}\}$ in $L^{2}_t L^2_x$ and $\{\nabla \mu^{\varepsilon}_{\lambda}\}$ in $L^{2}_t L^2_x$. It remains to pass to the limit $\varepsilon \to 0$ which is simple because \eqref{eq:approx_problem} admits a unique solution (it is nondegenerate Cahn-Hilliard with regular potential) so that we can identify the limit. 
\end{remark}


\section{Proof of the main result Theorem~\ref{thm:main}}
\label{sec:proof}

We consider $u_{\lambda_1}$, $u_{\lambda_2}$ to be two solutions of \eqref{eq:approx_problem} with potentials $\mu_{\lambda_1}$, $\mu_{\lambda_2}$ respectively. We want to prove
$$
\|u_{\lambda_1} - u_{\lambda_2}\|_{L^2(0,T; L^2(\Omega))}^2 \leq C(\lambda_1 + \lambda_2).
$$
To this end, we write $\overline{u}= u_{\lambda_1} - u_{\lambda_2}$, $\overline{\mu} = \mu_{\lambda_1} - \mu_{\lambda_2}$. We have
\begin{equation}\label{eq:equation_for_difference}
\partial_t \overline{u} = \Delta \overline{\mu}.
\end{equation}
Since the average of $\overline{u}$ is equal to $0$, we can test \eqref{eq:equation_for_difference} with $(-\Delta)^{-1} \overline{u}$ as in the proof of Lemma \ref{lem:L1_est_on_potential} to obtain
$$
\frac{1}{2} \, \partial_t  \int_{\Omega} \left|(-\Delta)^{-1/2} \overline{u}\right|^{2} + \int_{\Omega} \overline{\mu} \, \overline{u} = 0.
$$
Let $t \in [0,T]$. We integrate in time from $[0,t]$ to have
\begin{equation}\label{eq:H-1_energy_equality}
\frac{1}{2} \,  \int_{\Omega} \left|(-\Delta)^{-1/2} \overline{u}(t,\cdot)\right|^{2} + \int_0^t \int_{\Omega} \overline{\mu} \, \overline{u} = 0.
\end{equation}
The last term can be written as
$$
\int_0^t \int_{\Omega} \overline{\mu} \, \overline{u} = 
\int_0^t \int_{\Omega} B(\overline{u}) \, \overline{u} 
+ \int_0^t \int_{\Omega} (\gamma_{\lambda_1}(u_{\lambda_1}) - \gamma_{\lambda_2}(u_{\lambda_2})) \, \overline{u}
+ \int_0^t \int_{\Omega} (\Pi(u_{\lambda_1}) - \Pi(u_{\lambda_2})) \, \overline{u} 
$$
as $B$ is linear. We analyze the three terms appearing on the (RHS) separately.\\

\noindent \underline{Term $\int_0^t \int_{\Omega} B(\overline{u}) \, \overline{u}$.} We note that $B$ is of the form
\begin{equation}\label{eq:the_form_of_B}
Bu = a(x)\, u - Iu
\end{equation}
where $I$ is an integral Hilbert-Schmidt operator and self-adjoint because we assume that the kernel $J \in L^2_{\text{loc}}(\R^d)$. Therefore, $I$ is compact, and has representation
$$
If = \sum_{i=0}^{\infty} \langle If, e_i \rangle \, e_i
$$
where the orthonormal basis $\{e_i\}$ of $L^2(\Omega)$ is chosen as eigenvalues of $(-\Delta)$ operator as in \eqref{eq:Laplace_operator} and $\{\lambda_i\}$ are the corresponding eigenvalues. Note that $\lambda_0 = 0$, $e_0 = \text{const}$ and for all $i \geq 1$ we have $\int_{\Omega} e_i \diff x = 0$ and $(-\Delta)^{-1}e_i = \frac{1}{\lambda_i} e_i$.\\

As $I$ is Hilbert-Schmidt, $\sum_{i=1}^{\infty} \|Ie_i\|_{L^2(\Omega)}^2 <\infty$. Therefore, we conclude that there exists a sequence $
\{I_k\}$ of finite dimensional and self-adjoint operators defined by the formula
\begin{equation}\label{eq:finite_dim_operator_T}
I_kf = \sum_{i=0}^{k}  \langle If, e_i \rangle \,  e_i. 
\end{equation}
Moreover, $\|I_k - I\| \to 0$ in the operator norm which is a simple consequence of the summability $\sum_{i=1}^{\infty} \|Ie_i\|^2_{L^2(\Omega)} < \infty$.
Now, we fix $k \in \mathbb{N}$. We introduce the space
$$
L^2_0(\Omega) = \left\{f \in L^2(\Omega): \int_{\Omega} f(x) \diff x = 0\right\}
$$
and we use the orthogonal decomposition
$$
L^2_0(\Omega) = A_k \oplus B_k, \qquad \qquad \overline{u} = \overline{u}_A + \overline{u}_B, \qquad \qquad \overline{u}_A \in A_k,\qquad \overline{u}_B \in B_k,
$$
where $A_k = \mbox{span}(e_1,..., e_k)$ and $B_k = (e_{k+1}, e_{k+2},...)$. We remark that we skip the vector $e_0$ because we restrict ourselves to functions in $L^2_0(\Omega)$.\\

We write
$$
\int_0^t \int_{\Omega} I(\overline{u}) \, \overline{u} = \int_0^t\int_{\Omega} (I(\overline{u}) - I_k(\overline{u}))\, \overline{u} +
\int_0^t\int_{\Omega} I_k(\overline{u})\, \overline{u}.
$$
The first term is estimated by triangle inequality
$$
\int_0^t \int_{\Omega} (I(\overline{u}) - I_k(\overline{u}))\, \overline{u} \leq  \|I- I_k\| \, \|\overline{u}\|_{L^2(0,t; L^2(\Omega))}^2.
$$
For the second term, we use the decomposition $\overline{u} = \overline{u}_A + \overline{u}_B$ and observe that
$$
\int_0^t \int_{\Omega} I_k(\overline{u})\, \overline{u}_B = 0
$$
because $I_k$ has values in $A_k$. Therefore, 
$$
\int_0^t\int_{\Omega} I_k(\overline{u})\, \overline{u} = 
\int_0^t \int_{\Omega} I_k(\overline{u})\, \overline{u}_A \leq \varepsilon\, \|I\|^2\, \|\overline{u}\|^2_{L^2(0,t; L^2(\Omega))} + C(\varepsilon)\,\|\overline{u}_A\|_{L^2(0,t; L^2(\Omega))}^2,  
$$
because $\|I_k\| \leq \|I\|$. Now, on the space $A_k$ there are two norms: the usual $L^2$ norm and the $H^{-1}$ norm defined as $\|u\|_{H^{-1}(\Omega)} = \|(-\Delta)^{-1/2} u\|_{L^2(\Omega)}$ (this is a norm due to the additional constraint that the average is 0). As $A_k$ is of finite dimension, there exists a constant $C(k)$ (which depends on $k$ and blows up when $k\to \infty$) such that
$$
\|\overline{u}_A\|_{L^2(0,t; L^2(\Omega))}^2 \leq C(k) \, \| (-\Delta)^{-1/2} \overline{u}_{A} \|^2_{L^2(0,t; L^2(\Omega))} \leq C(k) \, \| (-\Delta)^{-1/2} \overline{u} \|^2_{L^2(0,t; L^2(\Omega))},
$$
where the last inequality follows from the fact that $\{e_i\}$ are the eigenvectors of $-\Delta$ operator and $(-\Delta)^{-1/2}$ is self-adjoint so that
$$
\| (-\Delta)^{-1/2} \overline{u}_{A} \|^2_{L^2(0,t; L^2(\Omega))} + \| (-\Delta)^{-1/2} \overline{u}_{B} \|^2_{L^2(0,t; L^2(\Omega))} = 
\| (-\Delta)^{-1/2} \overline{u} \|^2_{L^2(0,t; L^2(\Omega))}.
$$
The conclusion is that
\begin{multline*}
\int_0^t \int_{\Omega} I(\overline{u}) \, \overline{u}
    \leq 
    \\
    \leq \|I- I_k\| \, \|\overline{u}\|_{L^2(0,t; L^2(\Omega))}^2 + \varepsilon\, \|I\|^2\, \|\overline{u}\|^2_{L^2(0,t; L^2(\Omega))} + C(\varepsilon,k)\, \| (-\Delta)^{-1/2} \overline{u} \|^2_{L^2(0,t; L^2(\Omega))}. 
\end{multline*}
Taking into account the whole form of $B$ as in \eqref{eq:the_form_of_B} and that $a(x) > a_{-}$ as in \eqref{eq:ass_a}, we conclude that
\begin{equation}\label{eq:final_estimate_nonlocal}
\begin{split}
\int_0^t \int_{\Omega} &B(\overline{u}) \, \overline{u} \geq \\ 
&\geq 
\Bigl(a_{-} - \varepsilon\, \|I\|^2 - \|I-I_k\|\Bigr) \, \|\overline{u}\|_{L^2(0,t; L^2(\Omega))}^2  - C(\varepsilon,k)\,\| (-\Delta)^{-1/2} \overline{u} \|^2_{L^2(0,t; L^2(\Omega))},
\end{split}
\end{equation}
where $\varepsilon$ and $k$ has to be chosen appropriately.\\

\underline{Term $\int_0^t \int_{\Omega} (\gamma_{\lambda_1}(u_{\lambda_1}) - \gamma_{\lambda_2}(u_{\lambda_2})) \, \overline{u}$.} We use the  argument in~\cite{CG00}. Namely, directly by definition of Yosida approximation \eqref{eq:def_gamma_lambda}, we have $J_{\lambda} u = u - \lambda \, \gamma_{\lambda}(u)$ so that
\begin{multline*}
- \int_0^t \int_{\Omega} (\gamma_{\lambda_1}(u_{\lambda_1}) - \gamma_{\lambda_2}(u_{\lambda_2})) \, \overline{u} =
- \int_0^t\int_{\Omega} (\gamma_{\lambda_1}(u_{\lambda_1}) - \gamma_{\lambda_2}(u_{\lambda_2})) \, (J_{\lambda_1}(u_{\lambda_1}) -  J_{\lambda_2}(u_{\lambda_2}))\\
- \int_0^t \int_{\Omega} (\gamma_{\lambda_1}(u_{\lambda_1}) - \gamma_{\lambda_2}(u_{\lambda_2})) \, ( \lambda_1 \gamma_{\lambda_1}(u_{\lambda_1}) - 
\lambda_2 \gamma_{\lambda_2}(u_{\lambda_2})).
\end{multline*}
The first term appearing on the (RHS) is nonpositive by monotonicity of $\gamma$ and the fact that $\gamma_{\lambda}(u) \in \gamma(J_{\lambda}(u))$ cf. \eqref{eq:fact_about_gamma_lambda_in_the_resolvent}. Then, the remaining part can be estimated by Holder inequality:
\begin{align*}
- \int_0^t \int_{\Omega} (&\gamma_{\lambda_1}(u_{\lambda_1}) - \gamma_{\lambda_2}(u_{\lambda_2})) \, ( \lambda_1 \gamma_{\lambda_1}(u_{\lambda_1}) -
\lambda_2 \gamma_{\lambda_2}(u_{\lambda_2})) \leq \\
&\leq - \lambda_1\|\gamma_{\lambda_1}(u_{\lambda_1})\|_{L^2_{t,x}}^2 
- \lambda_2\|\gamma_{\lambda_2}(u_{\lambda_2})\|_{L^2_{t,x}}^2 + (\lambda_1 + \lambda_2) \int_0^t \int_{\Omega} \gamma_{\lambda_1}(u_{\lambda_1}) \,
\gamma_{\lambda_2}(u_{\lambda_2}) \\
&\leq (\lambda_1 + \lambda_2) \, \|\gamma_{\lambda_1}(u_{\lambda_1})\|_{L^2_{t,x}} \, 
\|\gamma_{\lambda_2}(u_{\lambda_2}) \|_{L^2_{t,x}} \leq C(\lambda_1 + \lambda_2). \phantom{\int_{\Omega}}
\end{align*}
Using \ref{estim5} in Theorem \ref{thm:uniform_estimates}, we conclude
\begin{equation}\label{eq:final_estimate_gamma}
- \int_0^t \int_{\Omega} (\gamma_{\lambda_1}(u_{\lambda_1}) - \gamma_{\lambda_2}(u_{\lambda_2})) \, ( \lambda_1 \gamma_{\lambda_1}(u_{\lambda_1}) -
\lambda_2 \gamma_{\lambda_2}(u_{\lambda_2})) \leq C(\lambda_1 + \lambda_2).
\end{equation}

\underline{Term $\int_0^t \int_{\Omega} (\Pi(u_{\lambda_1}) - \Pi(u_{\lambda_2}))\, \overline{u}$.}
Here, we simply estimate
\begin{equation}\label{eq:final_estimate_pi}
-\int_0^t \int_{\Omega} (\Pi(u_{\lambda_1}) - \Pi(u_{\lambda_2})) \, \overline{u} \leq \|\Pi'\|_{\infty} \|\overline{u}\|^2_{L^2(0,t; L^2(\Omega))}. 
\end{equation}
\underline{Conclusion of the proof.} Plugging estimates \eqref{eq:final_estimate_nonlocal}, \eqref{eq:final_estimate_gamma}, \eqref{eq:final_estimate_pi} into \eqref{eq:H-1_energy_equality}, we deduce
\begin{multline*}
\frac{1}{2} \,  \int_{\Omega} \left|(-\Delta)^{-1/2} \overline{u}(t,\cdot)\right|^{2} + (a_{-} - \varepsilon\, \|I\| - \|I-I_k\| - \|\Pi'\|_{\infty}) \, \| \overline{u} \|^2_{L^2(0,t; L^2(\Omega))} \leq \\ 
\leq C \,(\lambda_1+\lambda_2) + C(\varepsilon,k)\,\| (-\Delta)^{-1/2} \overline{u} \|^2_{L^2(0,t; L^2(\Omega))}.
\end{multline*}
We choose $\varepsilon$ and $k$ such that 
$$
a_{-} - \varepsilon\, \|I\| - \|I-I_k\| - \|\Pi'\|_{\infty} > 0.
$$
Then, by Gronwall's inequality, we obtain a first control on $\int_{\Omega} \left|(-\Delta)^{-1/2} \overline{u}(t,\cdot)\right|^{2}$ for all $t \in [0,T]$ and then
$$
\| u_{\lambda_1} - u_{\lambda_2} \|^2_{L^2(0,t; L^2(\Omega))} \leq C\, (\lambda_1 + \lambda_2).
$$
It follows that the sequence $\{u_{\lambda}\}$ is a Cauchy sequence in $L^2((0,T)\times\Omega)$ and we know from \cite[Section 4.3]{DST} that the limit solves \eqref{eq:nlch} in the weak sense. Hence, we may pass to the limit $\lambda_1 \to 0$ and obtain
$$
\| u_{\lambda_2} - u \|^2_{L^2(0,t; L^2(\Omega))} \leq C\, \lambda_2
$$
which completes the proof of Theorem~\ref{thm:main}.

\section*{Acknowledgements}

Piotr Gwiazda was supported by National Science Center, Poland through project no. 2018/30/M/ST1/00423. Jakub Skrzeczkowski was supported by National Science Center, Poland through project no. 2019/35/N/ST1/03459 and by the Advanced Grant Nonlocal-CPD (Nonlocal PDEs for Complex Particle Dynamics: Phase Transitions, Patterns and Synchronization) of the European Research Council Executive Agency (ERC) under the European Union’s Horizon 2020 research and innovation program (grant agreement No. 883363). The authors acknowledge the financial support by the University of Graz.

\bibliographystyle{abbrv}
\bibliography{ref}
\end{document}